\documentclass[12pt]{article}
\usepackage{amsmath,amsfonts,amsthm,amscd,amssymb,latexsym,comment}
\usepackage{enumerate,graphicx,psfrag,subfigure,changebar,oldgerm}
\usepackage[mathscr]{eucal}
\usepackage[usenames]{color}
\title{Automorphisms of Partially Commutative Groups I: Linear Subgroups
\footnote{Research carried out while the second and third named
authors were visiting Newcastle University with the support of
EPSRC  grant
EP/D065275/1, June 2005.}}
\author{ \textsf{Andrew J. Duncan}
 \and \textsf{Ilya V. Kazachkov}
 \and \textsf{Vladimir N. Remeslennikov}}

\def\nul{\emptyset }
\def\D{\Delta }

\def\e{\varepsilon }
\def\G{\Gamma }

\def\a{\alpha }
\def\t{\tau }

\newtheorem{thm}{Theorem}[section]
\newtheorem{lem}[thm]{Lemma}
\newtheorem{cor}[thm]{Corollary}
\newtheorem{prop}[thm]{Proposition}

\newtheorem{defn}[thm]{Definition}
\newtheorem*{defn*}{Definition}
\newtheorem{exam}[thm]{Example}

\numberwithin{equation}{section}
\numberwithin{figure}{section}
\newcommand{\Aut}{\operatorname{Aut}}
\newcommand{\ZZ}{\ensuremath{\mathbb{Z}}}
\newcommand{\QQ}{\ensuremath{\mathbb{Q}}}

\newcommand{\CC}{\ensuremath{\mathbb{C}}}

\newcommand{\az}{\mathop{{\ensuremath{\alpha}}}}
\newcommand{\la}{\langle}
\newcommand{\ra}{\rangle}
\newcommand{\maxx}{\texttt{max}}
\newcommand{\minn}{\texttt{min}}

\newcommand{\cl}{\operatorname{cl}}

\newcommand{\maps}{\rightarrow}

\newcommand{\bs}{\backslash}
\newcommand{\St}{\operatorname{St}}
\newcommand{\conj}{\operatorname{conj}}
\newcommand{\cSt}{\St^{\conj}}
\newcommand{\LInn}{\operatorname{Conj}}
\newcommand{\Inn}{\operatorname{Inn}}

\newcommand{\Tr}{\operatorname{Tr}}
\newcommand{\Trp}{\operatorname{Tr}_\perp}
\newcommand{\tr}{\operatorname{tr}}

\newcommand{\GL}{\operatorname{GL}}

\newcommand{\be}{\begin{enumerate}}
\newcommand{\ee}{\end{enumerate}}
\newlength{\nts}
\newlength{\rts}
\newlength{\lts}
\begin{document}

\maketitle

\begin{abstract}
To the Memory of Wilhelm Magnus.
\end{abstract}

 \tableofcontents

\subsection*{Glossary of Notation}
\begin{tabbing}
 $\G$ \hspace*{\lts}\= ---  \hspace*{\rts}\=\parbox{9.5cm}{a finite, simple,
    undirected graph with vertex set $X$}\\[\nts]
  $G$ or $G(\G)$ \> --- \> \parbox[t]{9.5cm}{the (free) partially commutative group
    with underlying graph $\G$}\\[\nts]
  $G(Y)$ \> --- \> \parbox[t]{9.5cm}{the subgroup of $G(\G)$ generated by $Y\subseteq X$} \\[\nts]
  $\lg(w)$ \> --- \> \parbox[t]{9.5cm}{the minimal length of a word $w'$ such that $w =_{G} w'$}\\[\nts]
  $u\circ v$ \> --- \> \parbox[t]{9.5cm}{$\lg(uv)=\lg(u)+\lg(v)$} \\[\nts]
  $\a(w)$ \> --- \> \parbox[t]{9.5cm}{$\{x\in X\mid x^{\pm 1}$ occurs in a word of
    minimal length representing $w\}$ } \\[\nts]
  $d(x,y)$ \> --- \> \parbox[t]{9.5cm}{the distance from $x$ to $y$, $x,y \in \G$}\\[\nts]
  $Y^\perp$ \> --- \> \parbox[t]{9.5cm}{the orthogonal complement of $Y$ in $X$} \\[\nts]
  $x\sim y$ \> --- \> \parbox[t]{9.5cm}{elements $x$ and $y$ of $X$ are equivalent:
    that is $x^{\perp}=y^{\perp}$}\\[\nts]
  $[x]$ \> --- \> \parbox[t]{9.5cm}{the equivalence class
    of $x$ under $\sim$}\\[\nts]
  $N_1$ \> --- \> \parbox[t]{9.5cm}{$\{x\in X\mid |[x]|=1 \}$}\\[\nts]
  $N_2$ \> --- \> \parbox[t]{9.5cm}{$\{x\in X\mid |[x]|\ge 2\}$ }\\[\nts]
  $X^\prime$ \> --- \> \parbox[t]{9.5cm}{the  quotient $X/\sim$}\\[\nts]
  $N_2^\prime $ \> --- \> \parbox[t]{9.5cm}{$\{[x]\in X^\prime \mid x\in N_2 \}$}\\[\nts]
  $\cl(Y)$ \> --- \> \parbox[t]{9.5cm}{the closure of $Y$ in $X$, i.e. $\cl(Y)=Y^{\perp \perp}$}\\[\nts]
  $L(\G)$ or $L$ \> --- \> \parbox[t]{9.5cm}{the lattice of closed sets of $\G$}\\[\nts]
  $x<_L y $ \> --- \> \parbox[t]{9.5cm}{$\cl(x)\subseteq \cl(y)$, $x,y\in X$}\\[\nts]
  $L_X$ \> --- \> \parbox[t]{9.5cm}{$\{Y\in L\mid Y=\cl(x)$ for some $x\in X\}$  }\\[\nts]
  $L^{\maxx} $ \> --- \> \parbox[t]{9.5cm}{the $<_L$-maximal elements of $L_X$}\\[\nts]
  $\prec$ \> --- \> \parbox[t]{9.5cm}{a total order on $X$}\\[\nts]
  $X^{\minn}$ \> --- \> \parbox[t]{9.5cm}{$\{x\in X\mid x$ is a $\prec$-minimal element of $[x]\}$ }\\[\nts]
  $Y^\minn$ \> --- \> \parbox[t]{9.5cm}{$X^\minn\cap Y$, where $Y\subseteq X$}\\[\nts]
  $\LInn(G) $ \> --- \> \parbox[t]{9.5cm}{the set of conjugating automorphisms of $G$}\\[\nts]
  $\St(L)$ \> --- \> \parbox[t]{9.5cm}{the stabiliser of all $G(Y)$ where $Y\in L$, i.e. automorphisms
    fixing all parabolic centralisers}\\[\nts]
  $\St(L_X)$ \> --- \> \parbox[t]{9.5cm}{$\{\phi\in \Aut(G):  G(Y)^\phi=G(Y) \textrm{ for all } Y\in L_X\}$}\\[\nts]
  $\cSt(L)$ \> --- \> \parbox[t]{9.5cm}{the conjugate-stabiliser of all $L$, i.e.
    automorphisms
    fixing all parabolic centralisers up to conjugation}\\[\nts]
  $\St(L^\maxx)$ \> --- \> \parbox[t]{9.5cm}{the stabiliser all subgroups $G(Y)$, for  $Y\in L^\maxx$}\\[\nts]
  $[\phi] $ \> --- \> \parbox[t]{9.5cm}{an integer valued matrix corresponding
    to an automorphism of $G$ } \\[\nts]
  $S_Y, U_Y, D_Y$ \> --- \> \parbox[t]{9.5cm}{sets of integer valued matrices corresponding to
    $Y\in L$ } \\[\nts]
  $L(Y) $ \> --- \> \parbox[t]{9.5cm}{elements of $L$ contained in $Y\in L$} \\[\nts]
 $\phi_Y $ \> --- \> \parbox[t]{9.5cm}{the restriction of $\phi\in \Aut(G)$ to $G(Y)$} \\[\nts]
  $\St_Y(L) $ \> --- \> \parbox[t]{9.5cm}{$\{\phi_Y\mid \phi\in \St(L)\}$, where $Y\subseteq X\}$} \\[\nts]
  $C_G(S)$ \> --- \> \parbox[t]{9.5cm}{the centraliser of a  subset $S$ of $G$}
\end{tabbing}
%%%%%%%%%%%%%%%5

\section*{Introduction}
Recently a wave of interest in groups of automorphisms of partially commutative groups has
risen; see \cite{VC1, VCC, VC2, GPR}. This emergence of interest may be  attributable  to the fact
that numerous geometric and arithmetic groups are subgroups of these groups.

The goal of this paper is to construct and describe certain arithmetic subgroups
of the automorphism group of a partially commutative group. More precisely,
given an arbitrary finite graph $\Gamma$ we construct an arithmetic
subgroup $\St(L(\G))$ (see Section \ref{sec:stab} for definitions), represented  as a
subgroup of $GL(n,\ZZ)$, where $n$ is the number of vertices of the graph $\Gamma$,
see Theorems \ref{thm:StLgen} and \ref{thm:key}. Note, that our proof is independent of the results of Laurence and Serviatius,
\cite{Laurence, Servatius}, which give a description of the generating set of the
automorphism group $\Aut(G(\G))$. One of the advantages of this proof is that it is largely
combinatorial, rather than group-theoretic, so could be adjusted to obtain analogous
results for partially commutative algebras determined by the graph $\Gamma$ (in various varieties of algebras).

In the last section of the paper we give a description of the decomposition
of the group $\cSt(L(\G))$
(see Section \ref{sec:stab} for definitions) as a semidirect
product of the group of conjugating automorphisms $\LInn(G)$
and $\St(L(\G))$. This  result is closely related to Theorem 1.4 of
\cite{GPR}, but the situations considered in \cite{GPR} and in this paper are somewhat different.

The authors are grateful to the Centre de Recerca Matem\`atica (CRM): these results were
announced at a workshop, supported  by the CRM,  in Manresa,
and the proof outlined in seminars, while the authors were on research visits
to the CRM in September 2006.

\section{Preliminaries}\label{sec:prelim}
\subsection{Graphs and Lattices of Closed Subsets }\label{sec:graph}
In this section we give definitions and a summary of the facts we need
concerning graphs, orthogonal systems and closed subsets of a graph. For further
details the reader is referred to \cite{DKR3}.
Graph will mean undirected, finite, simple graph throughout this paper. If $x$ and $y$ are vertices of a  graph then we define the {\em distance} $d(x,y)$ from $x$ to $y$ to be the minimum of the lengths of all paths from $x$ to $y$ in $\G$. A subgraph $S$ of a graph $\G$ is called a {\em full} subgraph if vertices $a$ and $b$ of $S$ are joined by an edge of $S$ whenever they are joined by an edge of $\G$.

Let $\G$ be a graph with $V(\G)=X$. A subset $Y$ of $X$ is called a {\em simplex} if the full subgraph of $\G$ with vertices $Y$ is isomorphic to  a
complete graph. Given a subset
$Y$ of $X$  the {\em orthogonal complement}
of $Y$ is defined to be
\[Y^\perp=\{u\in X|d(u,y)\le 1, \textrm{ for all } y\in Y\}.\]
By convention we set $\nul^\perp=X$. It is not hard to see that
$Y\subseteq Y^{\perp\perp}$ and $Y^\perp= Y^{\perp\perp\perp}$ \cite[Lemma 2.1]{DKR3}.
We define the {\em closure} of $Y$ to be
$\cl(Y)=Y^{\perp\perp}$.
The closure operator in $\G$ satisfies, among others, the properties that
$Y\subseteq \cl(Y)$,
$\cl(Y^\perp)=Y^\perp$ and
$\cl(\cl(Y))=\cl(Y)$ \cite[Lemma 2.4]{DKR3}.  Moreover if $Y_1\subseteq Y_2\subseteq X$
then $\cl(Y_1)\subseteq \cl(Y_2)$.
\begin{defn}
A subset $Y$ of $X$ is called {\em closed} {\rm(}with respect to $\G${\rm)}
if $Y=\cl(Y)$.
Denote by $L=L(\G)$ the set of all closed subsets of $X$.
\end{defn}

Then $\cl(Y)\in L$, for all $Y\subseteq X$ and $U\in L$ if
and only if $U=V^\perp$, for some $V\subseteq X$ \cite[Lemma 2.7]{DKR3}. The
relation $Y_1\subseteq Y_2$ defines a partial order on the set
$L$. As the closure operator $\cl$ is inclusion preserving and maps
arbitrary subsets of $X$ into closed sets,
$L$ is a lattice where
 the infimum  $Y_1\wedge Y_2$ of $Y_1$ and $Y_2$
is  $Y_1\wedge Y_2=\cl(Y_1\cap Y_2)=Y_1\cap Y_2$
and the supremum is
$Y_1 \vee Y_2=\cl(Y_1\cup Y_2)$.

\subsection{Partially Commutative Groups}
Let $\G$ be a finite, undirected, simple graph. Let $X=V(\G)$ be the set of vertices of $\G$ and let $F(X)$ be the free group on $X$. For elements $g,h$ of a group we denote the commutator $g^{-1}h^{-1}gh$ of $g$ and $h$ by $[g,h]$. Let
\[
R=\{[x_i,x_j]\in F(X)\mid x_i,x_j\in X \textrm{ and there is an edge from } x_i \textrm{ to } x_j
\textrm{ in } \G\}.
\]
We define the {\em partially commutative group with (commutation) graph } $\G$ to be the group $G(\G)$ with presentation $\left< X\mid R\right>$.
(Note that these are the groups which are called finitely generated free partially commutative groups in \cite{DK}.)

The subgroup generated by a subset $Y\subseteq X$ is called an {\em canonical parabolic subgroup} of $G$
and denoted  $G(Y)$. This subgroup is equal to the partially commutative group
with commutation graph the full subgraph of $\G$ generated by $Y$ \cite{B1}.
The connection between closed sets and the group $G(\G)$ is established by
Proposition 3.9 of \cite{DKR4}: a subgroup $G(Y)^g$ of $G$ is a centraliser if and only if $Y\in L(\G)$.
If $Y$ is a closed subset of $\G$ and $g\in G$ then the  subgroup $G(Y)^g=C_G(Y^\perp)^g$
is called an {\em parabolic centraliser}.

Let $\G$ be a simple graph, $G=G(\G)$ and let $w\in G$.
Denote by $\lg(w)$ the minimum of the lengths of words that represents the element $w$.
We say that   $w \in G$
 is {\em cyclically minimal} if and only if
$$
\lg(g^{-1}wg) \ge \lg(w)
$$
We write $u\circ v$ to express the fact that $\lg(uv)=l(u)+l(v)$. We say that
$u$ is a {\em left divisor} ({\em right divisor})  of $w$ if there exists $v$
such that $w=u\circ v$ ($w=v\circ u$). If $g\in G$ and $w$ is a word of minimal
length representing $w$ then we write $\a(g)$ for the set of elements $x\in X$
such that $x^\pm 1$ occurs in $w$. In \cite{EKR} it is shown that $\a(g)$ depends
only on $g$ and not on the choice of $w$.

The {\em non-commutation} graph of
the partially commutative  group $G(\G)$ is the graph $\Delta$,
dual to $\Gamma$, with vertex set $V(\Delta)=X$ and
an edge connecting $x_i$ and $x_j$ if and only if $\left[x_i, x_j
\right] \ne 1$. The graph $\D$ is a union of its connected
components $\D_1, \ldots , \D_k$ and  words that depend on
letters from distinct components commute. For any graph $\G$,
if $S$ is a subset
of $V(\G)$ we shall write $\G(S)$ for the full subgraph of $\G$ with
vertices $S$. Now, if
the vertex set of
$\D_k$ is $I_k$ and $\G_k=\G(I_k)$
then
$G=
G(\G_1) \times \cdots \times G(\G_k)$.
For $g \in G$ let $\az(g)$ be the set of elements $x$ of $X$ such
that $x^{\pm 1}$ occurs in a minimal word $w$ representing $g$.
It is shown in \cite{EKR} that $\a(g)$ is well-defined.
Now suppose that  the full subgraph $\Delta (\az(w))$
of $\Delta$ with vertices $\a(w)$ has connected components
$\D_1,\ldots, \D_l$ and let the vertex set of $\D_j$ be $I_j$.
Then, since $[I_j,I_k]=1$,
we can split $w$ into the product of commuting words,
$w=w_1\circ \cdots \circ w_l$, where $w_j\in G(\G(I_j))$, so $[w_j,w_k]=1$
for all $j,k$.
If $w$ is cyclically minimal then
we call this expression for $w$ a
 {\em block decomposition} of $w$ and
say $w_j$ is a {\em block} of $w$, for $j=1,\ldots ,l$. Thus $w$ itself is
 a block if and only if  $\D(\a(w))$ is connected.

As in \cite{DKR4} we make the following definition.
\begin{defn}
Let
$w$ be a
cyclically minimal root element of $G$ with block decomposition
$w=w_1 \cdots w_k$ and let $Z$ be a subset of $X$ such that
$Z\subseteq \a(w)^\perp$. Then the
subgroup $Q=Q(w,Z)=\langle w_1 \rangle \times \dots \times \langle
w_k\rangle \times G(Z)$ is called a
{\em canonical quasiparabolic subgroup} of $G$.
\end{defn}
A subgroup is called {\em quasiparabolic} if it is conjugate to a canonical quasiparabolic subgroup.
In \cite{DKR4} centralisers of arbitrary subsets of a partially commutative group are
characterised in terms of quasiparabolic subgroups and we shall use this result in Section \ref{subsec:L-order}.
%%%%%%%%%%%%%%%%%%5

\subsection{Conjugating Automorphisms}\label{sec:conj}
Automorphisms which act locally by conjugation play an important role in the structure of $\Aut(G)$.

For $S\subseteq X$ define $\G_S$ to be $\G\bs S$, the graph obtained from $\G$ by removing all vertices of $S$ and all their incident edges.
\begin{defn}\label{defn:elia}
Let $x\in X$ and let $C$ be a connected component of $\G_{x^\perp}$. The automorphism $\a_C(x)$ given by
\[
y\mapsto
\left\{
\begin{array}{ll}
y^x, &\textrm{ if } y\in C\\
y, &\textrm{ otherwise}
\end{array}
\right.
\]
is called an {\em elementary conjugating automorphism} of $\G$. The subgroup of $\Aut(G)$ generated by all elementary conjugating automorphisms (over all connected components of $\G_{x^\perp}$ and all $x\in X$) is called the group of {\em conjugating automorphisms} and is denoted $\LInn(G)$.
\end{defn}

\begin{thm}[M. Laurence \cite{Laurence}]\label{thm:laurence}
An $\psi\in \Aut(G)$ is a conjugating automorphism if there exists $g_x\in G$ such that $x^\phi=x^{g_x}$, for all $x\in X$.
\end{thm}
From Theorem \ref{thm:laurence} it follows that the group of inner automorphisms $\Inn(G)$ is a subgroup of $\LInn(G)$; and is therefore a normal subgroup.
%%%%%%%%%%%%%%%%%%%%%%%%%%%%%%%%%%%%%%%
\section{Stabilisers of Parabolic Centralisers}
\subsection{Stabiliser Subgroups} \label{sec:stab}
Throughout the remainder of this paper let $\G$ be a finite graph with vertices $X$, let $G=G(\G)$ and let $L=L(\G)$ be the lattice of closed sets of $\G$. We denote the automorphism group of $G$ by $\Aut(G)$.
\begin{defn}\label{defn:stab}
The {\em stabiliser} of $L$ is defined to be
\[
\St(L)=\{\phi\in \Aut(G): G(Y)^\phi=G(Y), \textrm{ for all } Y\in L\}
\]
and the {\em conjugate-stabiliser} of $L$ is defined to be
\begin{gather}\notag
\begin{split}
\cSt(L)=\{\phi\in \Aut(G):& \textrm{ there exists } g_Y \textrm{ such that } \\
&G(Y)^\phi=G(Y)^{g_Y}, \textrm{ for all } Y\in L\}.
\end{split}
\end{gather}
If $\phi\in \cSt(L)$, $Y\in L$ and $g_Y$ is such that $G(Y)^{\phi}=G(Y)^{g_Y}$ then we say that $\phi$ {\em acts as } $g_Y$ {\em on} $Y$.
\end{defn}

\begin{prop}\label{prop:stgp}
Both the stabiliser $\St(L)$ and the conjugate-stabiliser $\cSt(L)$ of $L$ are subgroups of $\Aut(G)$ and $\St(L)<\cSt(L)$.
\end{prop}
\begin{proof}
It is clear that the stabiliser of $L$  is a group and that it is contained in the conjugate-stabiliser. If $\phi\in \cSt(L)$ acts as $g_Y$ on
$Y\in L$ then $G(Y)= (G(Y)^{h})^\phi$, where $h=(g_Y^{-1})^{\phi^{-1}}$.
Thus $G(Y)^{\phi^{-1}}=G(Y)^{h}$. If $\phi^\prime \in \cSt(L)$  acts as $g^\prime_Y$ on $Y$ then $G(Y)^{\phi^\prime\phi}=G(Y)^{k}$, where
$k=g_Y(g_Y^\prime)^\phi$.
\end{proof}
%%%%%%%%%%%%%%%%%%%%%
\subsection{Generators for the Stabiliser of $L$}
We introduce three sets of maps $J$, $V_\perp$ and $\Trp$: which will turn out to be automorphisms and to generate $\St(L)$ (cf.\cite{Servatius}, \cite{Laurence}).
First we recall some notation from \cite{DKR3} and establish some background information.

We define an equivalence relation  $\sim$ on $X$ by
$x\sim y$ if and only if $x^\perp=y^\perp$. (Note that
this is the relation $\sim_\perp$ of
\cite{DKR3}.) Denote the equivalence
class of $x$ under $\sim$ by $[x]$.
Then $[x]$ is a simplex for all $x\in X$.
The set $N_2$ consists of those $x\in X$ such that $|[x]|\ge 2$.
Define $N_1=X\backslash N_2$, $X^\prime=X/\sim$ and
$N_2^\prime=\{[x]\in X^\prime| x\in N_2\}$.
If $x\in N_1$ then  $[x]=\{x\}$ so
 $X$ is the disjoint union
\[
X=\sqcup_{[x]\in N^\prime_2}[x] \sqcup N_1.
\]

For $x\in X$ we write $G[x]_\perp$ for $G([x]_\perp)=\la [x]_\perp\ra$, so $G[x]_\perp \subseteq G$. For all $x,y,z\in X$ such that $y\in [x]_\perp$ we have $[y,z]=1$ if and only if $[u,z]=1$, for all $u\in [x]_\perp$. It follows that we may extend an automorphism $\phi$ of $G[x]_\perp$ to an automorphism $\phi^{\e_x}$ of $G$, by setting $g^{\phi^{\e_x}}=g^\phi$, for all $g\in G[x]_\perp$ and $x^{\phi^\e_x}=x$, for all $x\in X\backslash [x]_\perp$. The map $\e_x$ such that $\phi\mapsto \phi^{\e_x}$ is then an monomorphism from $\Aut(G[x]_\perp)$ into $\Aut(G)$. Moreover, if $[x]_\perp\neq [y]_\perp$ then $G[x]_\perp\cap G[y]_\perp=1$. Therefore  $\Aut(G[x]_\perp)^{\e_x}\cap \Aut(G[y]_\perp)^{\e_y}=1$ and, for all $\phi\in \Aut(G[x]_\perp)$ and $\psi \in \Aut(G[y]_\perp)$, we have $\phi^{\e_x}\psi^{\e_y}=\psi^{\e_y}\phi^{\e_x}$. Hence $\Aut(G)$ contains a subgroup
\[
V= \prod_{[x]\in N_2^\prime}\Aut(G[x]_\perp)^{\e_x}\times \prod_{x\in N_1}\Aut(G[x]_\perp)^{\e_x}.
\]
\be
    \item The set $J$ consists of the extension to $G$ of all maps $X\maps G$ such that $x^\phi=x$ or $x^{-1}$,  for all  $x\in X.$ Then $J\le \Aut(G)$ and $|J|=2^{|X|}$.
    \item We define
\[
V_\perp=\prod_{[x]\in N_2^\prime}\Aut(G[x]_\perp)^{\e_x}\le V.
\]
    \item For distinct $x,y\in X$ define a map $\tr_{x,y}:X\maps G$ by $\tr_{x,y}(x)=xy$, and $\tr_{x,y}(z)=z$, for $z\neq x$. If $x^\perp\bs \{x\}\subseteq y^\perp$ then $\tr_{x,y}$ extends to an automorphism of $G$ which we also call $\tr_{x,y}$. Define $\Tr$ to be the set consisting of the extension to $G$ of all the maps $\tr_{x,y}$, where $x^\perp\bs\{x\} \subseteq y^\perp$. Then $\Tr\subseteq \Aut(G)$. We define
\[
\Trp=\{\tr_{x,y}\in \Tr: \cl(y)<\cl(x)\}\subseteq \Tr.
\]
    Note that we have excluded those $\tr_{x,y}$ where $\cl(y)=\cl(x)$ and that $\cl(y)<\cl(x)$ implies that $y\in \cl(x)$ so that $x^\perp\subset y^\perp$ and $[x,y]=1$. Therefore the map $\tr_{x,y}\in \Trp$ if and only if $x^\perp \subset y^\perp$ (the inclusion being strict).
\ee

We remark that the subgroup generated by $J$ and $V_\perp$ is $V$. Moreover
Laurence \cite{Laurence}, building on results of Servatius \cite{Servatius}, showed
that $\Aut(G)$ is generated by $J$, $\Tr$ and  $\LInn(G)$, together with automorphisms which
permute the vertices of $\G$ (see for example \cite{DKR3}).

\begin{prop}\label{prop:IVTrSt}
$J$, $V$  and  $V_\perp$ are subgroups of $\St(L)$ and the set $\Trp$ is contained in $\St(L)$.
\end{prop}
\begin{proof}
As $J$ fixes every parabolic subgroup it is clear that $J\le \St(L)$. To see that $V_{\perp}$ is a subgroup of $\St(L)$ consider $Y\in L$. If $x\in Y$ and $z\in [x]_\perp$ then $Y^\perp\subseteq  x^\perp=z^\perp$ so $z\in Y$. Hence $x\in Y$ implies that $[x]_\perp\subseteq Y$. Now suppose that $\phi\in V_\perp$. If $y\in Y\cap N_2$ then $y^\phi\in G[y]_\perp\le G(Y)$, by the above. If $y\in Y\bs N_2$ then $y^\phi=y\in G(Y)$. Hence $G(Y)^\phi
\subseteq G(Y)$ and since $V_\perp$ is a subgroup of $\Aut(G)$ it follows that $G(Y)^\phi=G(Y)$. Therefore $V_\perp\le \St(L)$. Finally, let $\t=\tr_{x,y}\in \Trp$ and let $Y\in L$. If $x\notin Y$ then $\t$ fixes $Y$ pointwise, so we assume that $x\in Y$. In this case $\cl(x)\subseteq Y$ and $y\in \cl(x)$, from the remark following the definition of $\Trp$. Hence $x^\t=xy\in G(Y)$ and $G(Y)^\t\le G(Y)$. As $x=(xy^{-1})^\t$ it follows that $G(Y)^\t=G(Y)$.  As $V$ is generated by $J$ and $V_\perp$ it is also a subgroup of $\St(L)$.
\end{proof}

Before stating the next theorem we shall briefly explain what is meant by an arithmetic group.
Two subgroups $A$ and $B$ of a group $G$ are said to be commensurable if $A\cap B$ is
of finite index in both $A$ and $B$.
A linear algebraic group is a group which is also an affine algebraic variety, such that
multiplication and inversion are morphisms of affine algebraic varieties. A linear algebraic
group is said to be $\QQ$-defined if it is a subgroup of $\GL(n,\CC)$ which can be defined by
polynomials over $\QQ$ and is such that the group operations are morphisms defined over $\QQ$. Let
$G$ be a $\QQ$-defined algebraic group. A subgroup $A\subseteq G\cap \GL(n,\QQ)$ is called an
{\em arithmetic subgroup} of $G$ if it is commensurable with $G\cap \GL(n,\ZZ)$. A group
is called {\em arithmetic} if it is isomorphic to an arithmetic subgroup of a $\QQ$-defined
linear algebraic group.

\begin{thm}\label{thm:StLgen}
The stabiliser $\St(L)$ is an arithmetic group, generated by the elements of  $J$, $V_\perp$ and $\Tr_\perp$.
\end{thm}
We defer the proof of  this theorem;  which is part of the more technical Theorem \ref{thm:key} below.
%%%%%%%%%%%%%%%%%%%%%%%%%%%%%%%%%%%%%%
\subsection{Ordering $L$ and the stabiliser of $L$-maximal elements}\label{subsec:L-order}
We shall now define a partial order on elements of $x$ which reflects the lattice
structure of $L$. We shall then describe a subgroup of the automorphism group of $G$ which
stabilises subgroups generated by closures of the maximal elements in this order. First
note that if $Y\in L$ then $Y=\cup_{y\in Y}\cl(y)$.
Therefore if $\phi\in \Aut(G)$ and $G(\cl(y))^\phi=G(\cl(y))$, for all $y\in Y$
then $G(Y)^\phi =G(Y)$. This implies that if $G(\cl(x))^\phi=G(\cl(x))$,
for all $x\in X$ then $\phi\in \St(L)$.
Setting  $L_X=\{Y\in L:Y=\cl(x), \textrm{ for some } x\in X\}$ and $\St(L_X)
=\{\phi\in \Aut(G):  G(Y)^\phi=G(Y) \textrm{ for all } Y\in L_X\}$ then we have
\begin{equation}\label{eq:L-red}
\St(L)=\St(L_X).
\end{equation}
\begin{defn}
Let $<_L$ be the partial order on $X$ given by $x<_L y$ if and
only if $\cl(x)\subset \cl(y)$ and $\cl(x)\neq \cl(y)$.
By $x=_L y$ we mean $\cl(x)= \cl(y)$. We say $x$ is $L$-{\em minimal} ($L$-{\em maximal})
if $y\le_L x$  ($x\le_L y$) implies $\cl(y)=\cl(x)$.
\end{defn}
Note that $y<_L x$ if and only if $\cl(y)\subseteq \cl(x)$ and $x\notin \cl(y)$.

Let $L^\maxx=\{\cl(x)\in L_X\mid x \textrm{ is } <_L\textrm{-maximal}\}$.
Denote by $\St(L^\maxx) =\{\phi \in \Aut(G)\mid G(Y)^\phi=G(Y) \hbox{ for all } Y\in L^\maxx\}$.
\begin{prop}\label{prop:max-comm}
  $\St(L)$ and $\St(L^\maxx)$ are commensurable.
\end{prop}
\begin{proof}
Since $\St(L)\subseteq \St(L^\maxx)$, it suffices to prove that $\St(L)$ has finite index in $\St(L^\maxx)$.
Let $\phi \in \St(L^\maxx)$, let $x\in X$, let $Z\in L^\maxx$, say $Z=\cl(z)$ for some $L$-maximal
element $z\in X$, and let $Y\in L$  such that $Y\subseteq Z$.
By  definition of $\St(L^\maxx)$ we have  $G(Z)^\phi=G(Z)$ and so
 $G(Y)^\phi\subseteq G(Z)^\phi =G(Z)$. As $Y\in L$ we have $Y=U^\perp$, for some $U\in L$, so $G(Y)=C_G(U)$, a
canonical parabolic centraliser. Hence $G(Y)^\phi$ is a centraliser and from \cite[Theorem 3.12]{DKR4}
is conjugate to a quasiparabolic subgroup. As $G(Z)$ is Abelian it follows (loc. cit.) that
$G(Y)^\phi$ is a canonical parabolic centraliser: that is $G(Y)^\phi=G(V)$ for some
$V\in L$ with $V\subseteq G(Z)$. As $\phi$ is an automorphism it therefore permutes the subgroups
$G(Y)$ where $Y$ is an element of
the set $L(Z)=\{V\in L: V\subseteq Z\}$ and this induces a permutation on the
set $L(Z)$. This holds for all elements $Z\in L^\maxx$ and setting $M=\cup_{Z\in L^\maxx} L(Z)$
we obtain a permutation $\sigma(\phi)$ of $M$.  It is clear that if $\psi$ is another element of
 $\St(L^\maxx)$ then $\sigma(\phi\psi)=\sigma(\phi)\sigma(\psi)$. Thus, writing
$S_M$ for the permutation group of the finite set $M$, we may view $\sigma$ as a homomorphism
$\sigma:\St(L^\maxx)\to S_M$.
If $\phi\in \ker(\sigma)$ then $\sigma(\phi)$ fixes every element of $M$ and in particular every
element of $L_X$. Hence $\ker (\sigma)\subseteq \St(L_X)$
and since from \eqref{eq:L-red} we have $\St(L_X)=\St(L)$ this completes the proof.
\end{proof}
%%%%%%%%%%%%%%%%%%%%%%%%%%%%%%%%%%%%%%
\subsection{Ordering $X$ and closures of elements}\label{subsec:order}
Using the order induced from $L$ we shall define a total order on $X$.
This order will be used to index a basis of $\ZZ^n$ with respect to which elements of
$\St(L)$ will be described as matrices. The ordering
 depends on the following stratification of the closures of single elements of $X$.
\begin{prop}\label{prop:bdry}
For all $x\in X$
\[
[x]_\perp = \cl(x)\bs \{u\in \cl(y): y<_L x\}.
\]
In particular, if $x$ is $L$-minimal then $[x]_\perp=\cl(x)$.
\end{prop}
\begin{proof}
First recall $\cl(z)=\cl(x)$ if and only if $z^\perp = x^\perp$ (as $Y^\perp=Y^{\perp\perp\perp}$), so $z\in [x]_\perp$ if and only if $\cl(z)=\cl(x)$.
If $u\in \cl(y)$, where $y<_L x$ then $\cl(u)\le \cl(y)<\cl(x)$ so $u^\perp \neq x^\perp$. Hence $[x]_\perp\subseteq \cl(x)\backslash\{u\in \cl(y): y<_L x\}$. On the other hand if $z\in \cl(x)$ then $\cl(z)\subseteq \cl(x)$. If also $z\notin \cl(y)$, for all $y<_L x$ then $z\not<_L x$, so $\cl(z)=\cl(x)$, as required.
\end{proof}
We  now define a total order $\prec$ on $X$, which will have the properties that
\be
    \item\label{it:prec1} if $x<_L y$ then $y\prec x$ and
    \item\label{it:prec2} if $z\prec y\prec x$ and $z\in [x]_\perp$ then $y\in [x]_\perp$.
\ee
To begin with let
\[
B_0=\{Y\in L_X:Y=\cl(x), \textrm{ where } x \textrm{ is } L\textrm{-minimal}\}.
\]
Suppose that $B_0$ has $k$ elements and choose an ordering
$Y_1<\cdots <Y_k$ of these elements.
If $i\neq j$ then $Y_i\cap Y_j\in L$ and from the remark at the beginning
of this section and the fact that the $Y_i$'s are $L$-minimal it follows
that $Y_i\cap Y_j=\emptyset$. Therefore we may
define the ordering $\prec$
on $\cup_{i=1}^k Y_i$ in such a way that if $x_i\in Y_i$ and $x_j\in Y_j$
and $Y_i<Y_j$ then $x_j\prec x_i$: merely by choosing an ordering for
elements of each $Y_i$.

We recursively define sets $B_i$ of elements of $L_X$, for $i\ge 0$,
as follows.
Assume that we have defined sets
$B_0,\ldots ,B_i$, set $U_i=\cup_{j=0}^i B_j$ and
define $X_i=\{u\in X: u\in Y, \textrm{ for some } Y\in U_i\}.$
If $U_i\neq L_X$
define $B_{i+1}$ by
\[B_{i+1}=\{Y=\cl(x)\in L_X:
 Y\notin U_i, \textrm{ and } y<_L x \textrm{ implies that } \cl(y)\in U_i
\}.\]
If $U_i\neq L_X$ then $X_i\neq X$ and $B_{i+1}\neq \nul$.
We assume inductively that we have ordered the set $X_i$ in such a way
that if $0\le a<b\le i$ then $x_a\in Y_a$ where $Y_a\in
B_a$ and $x_b\in Y_b$ where $Y_b\in B_b$ implies that
$x_b\prec x_a$.  From Proposition \ref{prop:bdry}, if $Y=\cl(x)\in B_{i+1}$
then
\begin{align*}
[x]_\perp
&=Y\bs \{u\in \cl(y):y<_L x\}\\
&=Y\bs \{u\in X_i\}.
\end{align*}
Therefore we have defined $\prec$ on the set $Y\bs [x]_\perp$.
Moreover, if $Y_1\neq Y_2$ and $Y_1,Y_2\in B_{i+1}$ then
$Y_1\cap Y_2\in L$ so $z\in Y_1\cap Y_2$ implies $\cl(z)\subseteq Y_1\cap Y_2$.
As $Y_1\neq Y_2$ this implies that $\cl(z)$ is strictly contained in
$Y_i$, $i=1,2$. If $Y_i=\cl(x_i)$ then $z<_L x_i$ and so $z\notin [x_i]_\perp$,
$i=1,2$. That is, $[x_1]_\perp\cap [x_2]_\perp=\emptyset$.
Now choose an ordering on the set of elements of $B_{i+1}$:
$Z_1<\cdots <Z_k$ say, where $Z_j=\cl(x_j)$.
Then $Z_j\bs [x_j]_\perp\subseteq X_i, j=1,\ldots ,k$.
We can extend the  total
order $\prec$ on $X_i$ to
\[
X_{i+1}=X_i\cup \cup_{j=1}^k Z_{j} =X_i\cup\cup_{j=1}^k[x_j]_\perp
\]
as follows.
Assume the order has already been extended to
$X_i\cup_{j=1}^{s-1} [x_{j}]_\perp$.
Extend the order further by choosing the ordering $\prec$ on
the elements of $[x_s]_\perp$ and then setting its greatest element
less
than the least element of $X_i\cup_{j=1}^{s-1}
[x_j]_\perp$.
At the final stage $s=k$ and the order on $X_i$ is extended to $X_{i+1}$.
We continue until $U_i=L_X$, at
which point $X=X_i$ and we have the required total order on $X$.
Note that, by construction,
if $x,y\in X$ and $x<_L y$ then
$y\prec x$. Also, if
$x\prec y\prec z$ and $[z]_\perp=[x]_\perp$ then $[y]_\perp =[x]_\perp$.
Thus \ref{it:prec1} and \ref{it:prec2} above hold.
If $\cl(x)$ belongs to $B_i$ we shall say that $x$, $\cl(x)$ and
$[x]_\perp$ have {\em height} $i$ and write $h(x)=h(\cl(x))=h([x]_\perp)
=i$.

\subsection{A matrix representation of $\St(L)$}\label{sec:matrix}
 Suppose that $X=\{x_1,\ldots ,x_k\}$ with $x_1\prec \cdots \prec x_k$.
If $x\in X$ and $\phi \in \St(L)$ then we have $x^\phi \in G(\cl(x))$. If
$\cl(x)=\{y_1,\ldots ,y_r\}$, where $x=y_1$ say, then as $G(\cl(x))$ is a
free Abelian group we may write
\begin{equation}\label{eq:phix}
x^\phi =y_1^{b_1}\cdots y_r^{b_r}.
\end{equation}
Setting $a_j=0$, if $x_j\notin \cl(x)$,
and $a_j=b_i$, if $x_j=y_i$, for some $i$, we can write
\[
x^\phi=x_1^{a_1}\cdots x_k^{a_k}.
\]
As this holds for all $x\in X$ we have
\begin{align*}
x_1^\phi& = x_1^{a_{1,1}} \cdots x_k^{a_{1,k}},\\
&\vdots\\
x_k^\phi &=  x_1^{a_{k,1}} \cdots x_k^{a_{k,k}}.
\end{align*}

Assume now that  $Y\in L$. Then $Y=\cup_{y\in Y} \cl(y)$ and,
 for $1\le i\le k$, either
$[x_i]_\perp\subseteq Y$ or $[x_i]_\perp \cap Y=\emptyset$.
Let $I=\{i:1\le i\le k\textrm{ and } [x_i]_\perp\subseteq Y\}$. Then,
from
Proposition \ref{prop:bdry}, it follows that
$Y=\cup_{i\in I}[x_i]_\perp$.
Moreover, for all $i$ such that $x_i\in Y$ we have
\[
x_i^\phi = \prod_{j\in I} x_j^{a_{i,j}}.
\]
We denote the restriction $\phi|_{G(Y)}$ of $\phi$ to $G(Y)$ by $\phi_Y$, for
any subset $Y$ of $X$ and $\phi\in \St(L)$. We shall also write $\phi_x$ instead of
$\phi_{\cl(x)}$, for $x\in X$.
\begin{defn}\label{defn:Mphi}
In the above notation, given $Y\in L$
we define the matrix
corresponding to the restriction $\phi_Y$ of $\phi\in \St(L)$ to
$G(Y)$
 to be
$[\phi_Y] =(a_{i,j})_{i,j\in I}$. If $Y=X$ we write $[\phi]$ for $[\phi_X]$.
\end{defn}
\begin{defn}\label{defn:minor}
Let $Y=\{y_1,\ldots ,y_r\}\in L$ with $y_1\prec \cdots \prec y_r$
and let $Z$ be a subset of $Y$.
Let $I=\{i:1\le i\le r,  y_i\in Z\}$.
Given $A=(a_{i,j})\in \GL(r,\ZZ)$ 
 we
define the $Z${\em -minor} of $A$ to be the matrix
$M(A,Y,Z)=(a_{i,j})_{i,j\in I}$. If $Y=X$ we write $M(A,Z)$ for $M(A,X,Z)$.
\end{defn}
The $Z$-minor of a matrix $A$ is
therefore the matrix obtained from $A$ by deleting
the $i$th
row and column for all $i$ such that $y_i\in Y\bs Z$.
From these definitions we have the following lemma.
\begin{lem}\label{lem:min}
Let $\phi\in \St(L)$ and $Y\in L$ then $[\phi_Y]=M([\phi],Y)$.
If $Z\subseteq Y\subseteq W$ are elements of $L$ then
$M(A,W,Z)=M(M(A,W,Y),Y,Z)$.
\end{lem}

For $x\in X$ we say that  $y\in [x]_\perp$ is the {\em minimal} element
of $[x]_\perp$ if $y\prec z$, for all $z\in [x]_\perp$.
Let $X^{\min}=\{x\in X: x \textrm{ is the minimal element of } [x]_\perp\}.$
Then $X=\cup_{x\in X}[x]_\perp=\sqcup_{x\in X^{\min}} [x]_\perp$. We extend
this notation to arbitrary $Y\in L$ by defining $Y^{\min} =X^{\min} \cap Y$;
so $Y=\sqcup_{y\in Y^{\min}} [y]_\perp$.

%%%%%%%%%%%%%%%
\subsection{Sets of Matrices Corresponding to Closed Sets}\label{subsec:SY}
We define a set of integer valued upper block-triangular matrices corresponding
to a closed set. Let $Y\in L$ and write $Y^{\min}=\{v_1,\ldots ,v_m\}$, where
$v_1\prec \ldots \prec v_m$.
Assume further that $Y=\{u_1,\ldots ,u_r\}$, where $u_1\prec \cdots \prec u_r$.
The set $S_Y$ is defined to the set of $r\times r$ integer valued
matrices
$A=(a_{i,j})$ such that the following conditions hold.
\be
\item\label{it:SY1} $A$ has $m$ diagonal blocks $A_1, \ldots ,A_m$, such that
$A_i\in \GL(|[v_i]_\perp|,\ZZ)$.
\item\label{it:SY2} If $i>j$ and $a_{i,j}$ is not part of a diagonal block then $a_{i,j}=0$.
\item\label{it:SY3} If $i<j$ and $a_{i,j}$ is not part of a diagonal block then
$a_{i,j}=0$ unless $u_j<_L u_i$; in which case $a_{i,j}$ may be any element of $\ZZ$.
\ee
The first two of these conditions imply that $A$ is an
upper block-triangular matrix.
Suppose that $A\in S_Y$ and has diagonal blocks $A_1,\ldots , A_m$.
Define the matrix $B$ to be the block-diagonal matrix with diagonal
blocks $A_1^{-1},\ldots ,A_m^{-1}$. Then $AB$ is a unipotent matrix
and it follows that $A\in \GL(r,\ZZ)$. Therefore $S_Y$ is a subset
of $\GL(r,\ZZ)$.

\begin{lem}\label{lem:Ssemi}\
\be
\item\label{it:Ssemi0} Elements of $S_Y$ are upper block-triangular
elements of $\GL(r,\ZZ)$.
\item\label{it:Ssemi1} If $\phi$ is an element of $J$, $V_\perp$ or $\Trp$ then
$[\phi]\in S_X$.
\item\label{it:Ssemi2}
The set $S_Y$ (with matrix multiplication) is a monoid, for all $Y\in L$.
\ee
\end{lem}
\begin{proof}
If $\phi\in J$ then $[\phi]$ is a diagonal matrix with diagonal entries
$\pm 1$, so belongs to $S_X$. If $\phi\in V_\perp$ then
$\phi= \prod_{[x]\in N_2^\prime}\phi_x^{\e_x}$, for some
automorphisms $\phi_x\in \Aut(G[x]_\perp)$. Hence $[\phi]$ is block
diagonal, with a blocks of dimension $|[x]_\perp|$, for each $[x]\in N_2^\prime$ and
of dimension $1$ for all $x\in N_1$. It follows that $[\phi]\in S_X$.
Finally let $\phi=\tr_{x,y}\in \Trp$. Then $\cl(y)<\cl(x)$ so $x\prec y$
and if $x=x_i$ and $y=x_j$ then $[\phi]$ is the matrix with $1$s on the
leading diagonal, $a_{i,j}=1$ and zeroes elsewhere.
As $x_i\prec x_j$ we have $i<j$ and, as $x_j<_L x_i$, $a_{i,j}$ is not
in a diagonal block so this matrix belongs to
$S_X$.  Thus statement \ref{it:Ssemi1} holds.

To prove statement \ref{it:Ssemi2}
assume that $Y=\{u_1,\ldots ,u_r\}$ where $u_1\prec \cdots \prec u_r$.
Let $A, B\in S_Y$, $A=(a_{i,j})$, $B=(b_{i,j})$ and let $AB=C=(c_{i,j})\in
\GL(r,\ZZ)$. Then $c_{i,j}=\sum_{k=1}^r a_{i,k}b_{k,j}$. Suppose that
$A$ has
diagonal blocks $A_1, \ldots, A_m$. Let $A_D$ be the block-diagonal matrix
with diagonal blocks $A_1, \ldots, A_m$ and zeroes elsewhere. Let
$A_N=A-A_D$. Define $B_D$ and $B_N$ similarly, so $B_D$ is block
diagonal and $B=B_D+B_N$. Then $C=A_DB_D + A_DB_N + A_N(B_D+B_N)$. Therefore
$C$ is upper block-triangular with diagonal blocks $A_1B_1, \ldots,
A_mB_m$ and $A_iB_i\in \GL(|A_i|,\ZZ)$.

Suppose that
$i<j$ and $c_{i,j}$ does not belong  to a diagonal block.
If $i$ and $k$ are such that $u_k\not\le_L u_i$ then
$a_{i,k}=0$ and  if $k,j$ are such that $u_j\not\le_L u_k$ then $b_{k,j}=0$.
Hence $a_{i,k} b_{k,j}\neq 0$ implies that  $u_k\le_L u_i$ and
$u_j\le_L u_k$.
If $u_i=_L u_j$ then
$c_{i,j}$ belongs to a diagonal block, a contradiction. Hence
$u_j<_L u_i$.  Thus $c_{i,j}\neq 0$ implies $u_j<_L u_i$ and so $C\in S_Y$.
As the identity matrix is in $S_Y$ it follows that $S_Y$ is a monoid.
\end{proof}

We are now ready to prove Theorem \ref{thm:StLgen}, which is
the second statement of the following Theorem.
\begin{thm}\label{thm:key}
The map $\pi:\St(L)\maps \GL(|X|,\ZZ)$ given by $\phi\mapsto [\phi]$  is an injective homomorphism with image $S_X$. In particular $S_X$ is a group. Moreover the group $\St(L)$ is generated by the elements of $J$, $V_\perp$ and $\Trp$, and is an arithmetic group.
\end{thm}
\begin{proof}
We shall first show that $\pi$ is an injective group homomorphism, from $\St(L)$ to $\GL(|X|,\ZZ)$. Assume that $|X|=k$, let $\phi,\psi\in \St(L)$ and
let $[\phi]=(a_{i,j})$ and $[\psi]=(b_{i,j})$. In the notation of Section \ref{subsec:order}, for $x_i\in X$ we have $x_i^{\phi\psi}=(x_1^{a_{i,1}})^\psi\cdots (x_k^{a_{i,k}})^\psi =\prod_{r=1}^k(x_1^{b_{r,1}}\cdots x_k^{b_{r,k}})^{a_{i,r}}=\prod_{j=1}^kx_j^{c_{i,j}}$, where $c_{i,j}=\sum_{r=1}^k a_{i,r}b_{r,j}$. Hence $[\phi\psi]=(c_{i,j})=[\phi][\psi]$. Therefore $\pi$ is a homomorphism and it is immediate from the definition that $\pi$ is injective.

Let $T$ denote the subgroup of $\St(L)$  generated by $J$, $V_\perp$ and $\Trp$  and recall that $V$ is the subgroup of $\St(L)$ generated by $J$ and $V_\perp$. From Lemma \ref{lem:Ssemi} it follows that $[\phi]\in S_X$, for all $\phi \in T$. Therefore once we have proved the final statement of the Lemma it will also follow that the image of $\pi$ is contained in $S_X$. The proof of the final statement will be broken into three cases and in each case we shall also verify that $\pi$  maps surjectively onto $S_X$.

Let $\phi\in \St(L)$ and $x\in X$ and suppose that $\cl(x)= \{y_1,\ldots ,y_r\}$. Then we can express $x^\phi$ as in \eqref{eq:phix}.Assume further that $[x]_\perp=\{y_1,\ldots ,y_s\}$, where $s\le r$. In the notation of \eqref{eq:phix}, if $b_j=0$, for all $j>s$, then $x^\phi\in G[x]_\perp$. Suppose this holds for all $x\in [x]_\perp$; so $\phi_{[x]_\perp} \in \Aut(G[x]_\perp)$. In this case we call $\phi$ a {\em block-diagonal} automorphism.
\subsubsection*{Case 1} \label{subsub:case1}
Let $\phi\in \St(L)$ be a block-diagonal automorphism and let $x\in X$. If $x\in N_2$ then $[x]\in N_2^\prime$ and $\phi_{[x]_\perp}^{\e_x} \in \Aut(G[x])^{\e_x}\subseteq V_\perp$. If $x\notin N_2$ then $[x]_\perp=\{x\}$ and $\phi_{[x]_\perp}^{\e_x}\in J$. In either case $\phi_{[x]_\perp}^{\e_x}\in V$ and, as the same is true of all $x\in X$,
\[
\phi =\prod_{[x]\in N_2^\prime}\phi_{[x]_\perp}^{\e_x} \cdot\prod_{x\in N_1}\phi_{[x]_\perp}^{\e_x}  \subseteq V\subseteq T.
\]

Now let $X=\{x_1,\ldots ,x_k\}$, where $x_1\prec \cdots \prec x_k$ and write $[\phi]=(a_{i,j})$. Let $X^{\min}=\{x_{i_1},\ldots ,x_{i_m}\}$, for some $m\ge 1$, with $x_{i_1}\prec \cdots \prec x_{i_m}$.  In this terminology what we have shown is the following.
\begin{gather}%\label{eq:Qcond}
\textrm{
If
 }
%$
a_{i,j}=0,
%$
\textrm{
for all
  }
%$
i,j
%$
\textrm{
such
that
  }
%$
x_i\in [x_{i_n}]_\perp,
x_j\notin [x_{i_n}]_\perp,
\textrm{
then
  }
%$
\phi_{[x_{i_n}]_\perp}^{\e_{i_n}}\in V, \notag\\
\textrm{ and if this holds for }
n=1,\ldots ,m,
\textrm{
then
}
%\textrm{ for }
%n=1,\ldots ,m,
%\textrm{ and }
\phi= \prod_{n=1}^m\phi_{[x_{i_n}]_\perp}^{\e_{i_n}} \in V\subseteq T.
\label{eq:Qcond}
%$
\end{gather}
That is, if \eqref{eq:Qcond} holds then $\phi\in T$; and so $[\phi] \in S_X$.

On the other hand, let $A\in S_X$ be a block-diagonal matrix. Then $A$ has  diagonal blocks $A_n\in \GL(|[x_{i_n}]_\perp|,\ZZ)$, for $n=1,\ldots ,m$.
Here $A_n$ determines an automorphism, $\phi_n$ say, of $G[x_{i_n}]_\perp$,  and $\phi=\prod_{n=1}^m \phi_n^{\e_{x_{i_n}}}\in V$. Moreover $[\phi]=A$; so all block-diagonal matrices in $S_X$ are in the image of $\pi$.
\subsubsection*{Case 2}\label{subsub:case2}
Let $\phi\in \St(L)$ and $A=[\phi]$. Write $A=A_D+A_N$ as in the proof of Lemma \ref{lem:Ssemi}. In this case we assume that $A_D$ is the identity matrix.  This means that $x_i^\phi =x_iw_i$, where $w_i\in G(Y)$, for some $Y\subseteq \cl(x) \backslash [x_i]_\perp$, for $i=1,\ldots , k$.
Define $r=r(\phi)$ to be the maximal integer such that $x_r\in [x_{i_n}]_\perp$, for some $n$, and there exists $x_j$ such that $x_j\notin [x_{i_n}]_\perp$ but $a_{r,j}\neq 0$. Let $j=c=c(\phi)$ be maximal with this property. (The argument of Case 1 covers the case $r=0$.) As $a_{r,c}\neq 0$ we have $w_r=u x_c^{a_{r,c}}$, for some $u\in G(Z)$, where $x_r\notin Z$ and since $A_D$ is the identity $w_s\in G(Z_s)$, where $x_r\notin Z_s$, for
all $s>r$.

As $\phi\in \St(L)$ we have $x_r^\phi \in \cl(x_r)$ so $x_c\in \cl(x_r)$ which implies $\cl(x_c)\subseteq \cl(x_r)$. Since $x_c\notin[x_r]_\perp$ it follows that $\cl(x_c)\neq \cl(x_r)$. Hence $\tr_{x_r,x_c}\in \Trp\subseteq \St(L)$. Let $\phi_1=(\tr_{x_r,x_c})^{-a_{r,c}}\in T$. Then $x_r^{\phi_1}=x_rx_c^{-a_{r,c}}$ and $x_l^{\phi_1}=x_l$, for all $l\neq r$. Let $\phi_0= \phi\phi_1$; so $\phi_0\in \St(L)$. We have
\[
x_r^{\phi_0}=(x_rux_c^{a_{r,c}})^{\phi_1}=x_ru \textrm{ and } x_s^{\phi_0}=(x_sw_s)^{\phi_1}=x_sw_s=x_s^{\phi},
\]
for $s>r$. If $s<r$ then
\[
x_s^{\phi_0}=x_sw_s^{\phi_1}=
\left\{
\begin{array}{ll}
x_sw_sx_c^{a_{r,c}} =x_s^\phi x_c^{a_{r,c}}, & \textrm{ if } x_r \textrm{ occurs in }w_s\\
x_sw_s=x_s^\phi, & \textrm{ otherwise}.
\end{array}
\right.
\]
Therefore all diagonal blocks of $[\phi_0]$ are the identity matrix and either $r([\phi_0])<r$ or $c([\phi_0])<c$. We may then assume inductively that $[\phi_0]\in S_X$ and $\phi_0\in T$: so $\phi\in T$. Now define $E$ to be the matrix which has zero in every position except row $r$ column $c$, which is equal to $a_{r,c}$. Then $[\phi_1^{\pm 1}]=I\mp E\in S_X$ and from Lemma \ref{lem:Ssemi} it follows that $[\phi]=[\phi_0][\phi_1^{-1}]\in S_X$. By induction it follows that for all $\phi$ such that $[\phi]=A=A_D+A_N$, with $A_D$ the identity, we have $\phi\in T$ and $\phi^\pi\in S_X$. The same argument shows that if $A\in S_X$ and $A=A_D+A_N$, with $A_D$ the identity, then $A=\phi^\pi$, for some $\phi\in T$.
\subsubsection*{Case 3}
In the general case let $\phi\in \St(L)$ and write $[\phi]=A =A_D+A_N$ as before. Let $B=A_D^{-1}$. Then from Case 1, $B^{\pm 1}\in S_X$ and $B=[\sigma_B]$, for some $\sigma_B\in T$. Let $\zeta=\phi\sigma_B$. Then all diagonal blocks of $[\zeta]$ are the identity, so $\zeta \in T$ and $[\zeta]\in S_X$, from Case 2. Therefore $\phi\in T$ and $[\phi]\in S_X$. If we begin this argument with an arbitrary element $A$ of $S_X$ instead
of an element of $\St(L)$ it shows again that $A\in T^\pi$ and $A\in S_X$.

We now show that the group $\St(L)$ is an arithmetic group.  Let $K$ be the subgroup of
$\GL(|X|,\CC)$ satisfying
conditions (\ref{it:SY1}),(\ref{it:SY2})  and (\ref{it:SY3}) in the definition of $S_X$
above. Then $K$ is a $\QQ$-defined linear algebraic group. As
$S_X=K\cap GL(|X|,\ZZ)$ it now follows that $\St(L)$ is arithmetic.
\end{proof}

Combining the final statement of the theorem with Proposition \ref{prop:max-comm} we obtain  the
following corollary.
\begin{cor}\label{cor:max-arith}
$\St(L^\maxx)$ is an arithmetic group.
\end{cor}
\begin{proof}
From the proof of  Proposition \ref{prop:max-comm}
$\St(L)$ is finite index subgroup of $\St(L^\maxx)$.
Also the proof of
of Theorem \ref{thm:key}  goes through to show that
$\St(L^\maxx)$ is isomorphic to a subgroup
of $S^\maxx$ of $\GL(|X|,\ZZ)$. To see this,
for each $x\in X$ let
$M_x=\{z\in X: z \textrm{ is } L\textrm{-maximal and } x<_L z\}$.
Then $G(M_x)$ is Abelian and contains $x^\phi$, for all $\phi\in \St(L^\maxx)$.
Let $\phi\in \St(L^\maxx)$ and, as at the begining of Section \ref{sec:matrix}, write
$x^\phi=x_1^{a_1}\cdots x_k^{a_k}$, where $a_{j}\neq 0$ only if $x_j\in M_x$. As before this allows us
to associate an integer valued  matrix $[\phi]$ to $\phi$. The proof that the map $\phi\mapsto [\phi]$
is a monomorphism from $\St(L^\maxx)$ into  $\GL(|X|,\ZZ)$ is exactly the same as the first paragraph
of the proof of Theorem \ref{thm:key}. Thus $\St(L^\maxx)$ is isomorphic to
its image $S^\maxx\subseteq \GL(|X|,\ZZ)$. Moreover this monomorphism restricts to
$\St(L)$ to give the map $\pi$ and
so $S_X$ is a finite index subgroup of $S^\maxx$.

Keeping the notation of the proof of the previous theorem we have $S_X=K\cap \GL(|X|,\CC)$.
Now choose a transversal $a_1,\ldots ,a_s$ for cosets of $S_X$ in $S^\maxx$. Then
$g\in S^\maxx$ if and only if $ga_r^{-1}\in S_X\subseteq K$, for some $r$. As
$S^\maxx\in \GL(|X|,\ZZ)$ so $a_r^{-1}\in \GL(|X|,\ZZ)$, for all $r$. Hence the
condition that an element $h\in \GL(|X|,\CC)$ satisfies $h=ga_r^{-1}$, for some $g$,
can be expressed using $|X|^2$ polynomials with integer coefficients (namely the
entries of the matrix $a_r^{-1}$). Set $p=|X|$ and let these polynomials
be $m_{r,1,1}, \ldots, m_{r,p,p}$. (Thus if
$h=ga_r^{-1}=(h_{ij})$ then substitution of entries of $g$ for variables
of the $m_{r,i,j}$ gives $h_{ij}=m_{r,i,j}(g)$, for all $i,j$.) Suppose
 that the algebraic variety $K$ is defined by
polynomials $f_1,\ldots, f_l$. Then $g\in Ka_r$ if and only if $g$ satisfies the
polynomial equations $f_i(m_{r,1,1},\ldots, m_{r,p,p})$, $i=1,\ldots ,l$. As $f_i$ and
$m_{r,i,j}$ are polynomials with integer coefficients this implies that $Ka_r$ is
a $\QQ$-defined affine algebraic variety.
Thus $\cup_{r=1}^s Ka_r$ is a variety and
\[\left( \cup_{r=1}^s Ka_r\right)\cap \GL(|X|,\ZZ)=
\cup_{r=1}^s\left(K\cap \GL(|X|,\ZZ)\right)a_r
=\cup_{r=1}^s S_X a_r=S^\maxx
\]
so $\St(L^\maxx)$
is an arithmetic group.
\end{proof}

In the previous theorem we restricted attention to the entire group $\St(L)$ and
its isomorphic image $S_X$. However, we shall now show, the set $S_Y$ is a group,
for all closed sets $Y$ in $L$, and in fact all these groups are arithmetic.
By defining appropriate maps corresponding to inclusion, as follows, it can be seen that
the lattice $L$  maps contravariantly
to a sublattice of the lattice of subgroups of $\Aut(G)$.
If $Y,Z\in L$ with $Y\subseteq Z$ then $M(A,Y,Z)\in \GL(|Z|,\ZZ)$
and so $\rho(Y,Z):A\mapsto M(A,Y,Z)$ is a map from $S_Y$ to $\GL(|Z|,\ZZ)$.
\begin{lem}\label{lem:Mhom}
Let $Z, Y\in L$ with $Z\subseteq Y$.
The set $S_Y$ is an arithmetic group and
the map $\rho(Y,Z)$ is a surjective
homomorphism from $S_Y$ to $S_Z$. There is an injective
homomorphism $\e(Z,Y)$ from $S_Z$ to $S_Y$ such that
$\e(Z,Y)\rho(Y,Z)$ is the identity on $S_Z$.
\end{lem}
\begin{proof}
We show that $\rho(Y,Z)$ is an surjective monoid homomorphism,
for all $Z\subseteq Y\in L$. Since $S_X$ is a group it will then
follow that $S_Y$ is a group, for all $Y\in L$. The proof that $S_Y$ is
arithmetic is then the same as for $S_X$, replacing $X$ by $Y$ throughout.
Let $Y^{\min}=\{v_1,\ldots ,v_m\}$, where $v_1\prec \cdots \prec v_m$.
Also let $Y=\{u_1,\ldots ,u_r\}$, where $u_1\prec \cdots \prec u_r$ and
let $I=\{i:1\le i\le r\textrm { and }u_i\in Z\}$.
Let $A=(a_{i,j})\in S_Y$ and suppose that $A$ has diagonal blocks
$A_1,\ldots ,A_m$.
As $A$ is upper block-triangular it follows from the definition that
$M(A,Y,Z)$ is upper block-triangular.
If $v_i\in Z$ then $[v_i]_\perp \subseteq Z$ and
the diagonal block containing $A_i$ is unaffected in the transformation
of $A$ to
$M(A,Y,Z)$. On the other hand if $v_i\notin Z$ then the diagonal
block $A_i$ is deleted in forming $M(A,Y,Z)$.
As $Z^{\min}=Y^{\min}\cap X$, the diagonal blocks of $M(A,Y,Z)$
satisfy condition \ref{it:SY1} of the definition of $S_Z$.

It remains to verify condition \ref{it:SY3}. Suppose that $i,j\in I$ and
that
$a_{i,j}\neq 0$ and $a_{i,j}$ does not belong to a diagonal block of
$M(A,Y,Z)$. From the above, $a_{i,j}$ does not belong to a
diagonal block of $A$, and
since $A$ is upper block-triangular and satisfies condition \ref{it:SY3},
$i<j$ and $u_j<_L u_i$.
Then
the same holds for $M(A,Y,Z)$, as required.
Therefore $\rho(Y,Z)$ maps $S_Y$ into $S_Z$.

To see that $\rho(Y,Z)$ is a homomorphism let $A=(a_{i,j})$ and $B=(b_{i,j})$
be elements of $S_Y$ and let $C=(c_{i,j})=AB$. From Lemma \ref{lem:Ssemi}, we
have $C\in S_Y$. Suppose that $i,j\in I$ and that $a_{i,k}b_{k,j}\neq 0$,
for some $k$. Then $u_k\le_L u_i$ and $u_i\in Z$. Hence $u_k\in \cl(u_k)
\subseteq \cl(u_i)\subseteq Z$. Therefore $i,j,k\in I$ and
$c_{i,j}=\sum_{k\in I} a_{i,k}b_{k,j}$. It follows that
$M(AB,Y,Z)=M(A,Y,Z)M(B,Y,Z)$, so $\rho(Y,Z)$ is a homomorphism.

To construct $\e(Z,Y)$ note that if $P\in S_Z$ then we may write
$P=(p_{i,j})_{i,j\in I}$, by writing $Z$ as a subset of $\{u_1,\ldots ,u_r\}$.
Then let the diagonal blocks of $P$ be $P_i$, where $i\in I$.
With this notation define an $r\times r$ integer matrix
$A$ by first of all setting $a_{i,j}=p_{i,j}$, for $i,j\in I$; then setting
$a_{i,i}=1$, for $i\notin I$, and finally setting $a_{i,j}=0$ for all
other $i,j$. Then $A$ is upper block-triangular and has blocks
$A_1,\ldots ,A_m$, where $A_i=P_i$, if $i\in I$ and $A_i$ is the identity
matrix in $\GL(|[v_i]_\perp|,\ZZ)$, otherwise. As $P$ satisfies
condition \ref{it:SY3} then so does $A$. Hence $A$ belongs to $S_Y$.
Define $M(P,Z,Y)=A$ (where $Z\subseteq Y$) and
$P^{\e(Y,Z)}=M(P,Z,Y)$, for all $P\in S_Z$. By definition
$P^{\e(Y,Z)\rho(Y,Z)}=P$, for all $P\in S_Z$, so $\e(Y,Z)$ is injective
and $\rho(Y,Z)$ is surjective. Moreover, from the definition,
$\e(Y,Z)$ is a homomorphism.

\end{proof}
%%%%%%%%%%%%%%%%%%%%%%%%%%%%%%%
\subsection{Restriction to closed sets}
Here we consider the restriction of automorphisms
in $\St(L)$ to subgroups $G(Y)$, where $Y$ is in $L$.
Given $Y\in L$ we
define $L(Y)=\{Z\in L: Z\subseteq Y\}$. Note that
$L(Y)$ is not in general the same as $L(\G(Y))$, the set of closed sets
of the full subgraph $\G(Y)$ of $\G$ generated by $Y$; although
$L\G(Y))\subseteq L(Y)$. We define $\St_Y(L)=
\{\phi_Y: \phi \in L \}$. Then $\St_Y(L)$ is a subgroup
of $\Aut(G(Y))$ and
is contained in the
subgroup of stabilisers, in $\Aut(G(Y))$, of $L(Y)$.

\begin{lem}\label{lem:StLY}
The
 map $\rho_Y:\St(L)\maps \St_Y(L)$ given by
$\phi\mapsto \phi_Y$ is a surjective homomorphism.
 The map
$\pi_Y:\St_Y(L)\maps S_Y$ given by $\phi_Y\mapsto [\phi_Y]$
 is an
isomorphism, so $\St_Y(L)$ is arithmetic,
for all $Y\in L$. Moreover  $\rho_Y\pi_Y=\pi_X\rho(X,Y)$.
\end{lem}
\begin{proof}
Let $\phi ,\psi \in \St(L)$. Then
$x^\phi \in G(Y)$, for all $x\in Y$, so $x^{\phi\psi}=(x^\phi)^\psi=
(x^{\phi_Y})^{\psi_Y}$, for all $x\in Y$. Hence $(\phi\psi)_Y=
\phi_Y\psi_Y$ and $\rho_Y$ is a homomorphism;  surjective onto its image
which is, by definition,
 $\St_Y(L)$.

From Theorem \ref{thm:key} the map $\pi_X=\pi$ is an isomorphism from
$\St(L)$ to $S_X$. From Lemma \ref{lem:Mhom} the map $\rho(X,Y)$ is
a surjective homomorphism from $S_X$ to $S_Y$. Let $\theta =\pi_X\rho(X,Y)$.
An element $\phi\in \St(L)$ belongs to $\ker(\rho_Y)$ if $x^\phi=x$, for
all $x\in Y$: in which case $[\phi_Y]$ is the identity matrix of
dimension $|Y|$. Hence the diagonal blocks of $\phi^{\pi_X}$ corresponding
to $[x]_\perp\subseteq Y$ are identity matrices; and $\phi^\theta=I$, the
 $|Y|$-dimensional identity matrix. This shows  that $\ker(\rho_Y)
\subseteq \ker(\theta)$, so $\theta$ induces a homomorphism from $\St_Y(L)$
to $S_Y$. The image of $\phi_Y$ under this homomorphism is
$\phi^{\theta}=[\phi]^{\rho(X,Y)}=M([\phi],Y)$ and from the definitions
we have $\phi_Y^{\pi_Y}=M([\phi],Y)$. Therefore $\pi_Y$ is a homomorphism
and $\rho_Y\pi_Y=\pi_X\rho(X,Y)$.
As $\theta$ is surjective, so is $\pi_Y$.  If $\phi_Y^{\pi_Y}=I$ then
$x^{\phi_Y}=x$, for all $x\in Y$, so $\phi_Y$ is the identity on $G(Y)$ and
$\pi_Y$ is injective.
\end{proof}
If $Z\subseteq Y\in L$ and $\phi\in \St(L)$ we define $\rho_{Y,Z}$ to
be the map sending $\phi\in \St_Y(L)$ to $\phi|_{G(Z)}\in \St_Z(L)$.
\begin{cor}
Let $Y,Z\in L$ with $Z\subseteq Y$. Then $\rho(X,Z)=\rho(X,Y)\rho(Y,Z)$
 and $\rho_Z=\rho_Y\rho_{Y,Z}$. Moreover  $\rho_{Y,Z}$ is surjective
and $\pi_Y\rho(Y,Z)=\rho_{Y,Z}\pi_Z$.
\end{cor}
\begin{proof}
This follows from Lemmas \ref{lem:min}, \ref{lem:Mhom} and \ref{lem:StLY}.
\end{proof}

The various maps we have defined are illustrated in the commutative
diagram of Figure \ref{fig:maps}.
\begin{figure}
\begin{center}
\psfrag{St(L)}{$\St(L)$}
\psfrag{Sty(L)}{$\St_Y(L)$}
\psfrag{Stz(L)}{$\St_Z(L)$}
\psfrag{Sx}{$S_X$}
\psfrag{Sy}{$S_Y$}
\psfrag{Sz}{$S_Z$}
\psfrag{py}{$\rho_Y$}
\psfrag{pz}{$\rho_Z$}
\psfrag{pyz}{$\rho_{Y,Z}$}
\psfrag{Pxy}{$\rho(X,Y)$}
\psfrag{Pyz}{$\rho(Y,Z)$}
\psfrag{Pxz}{$\rho(X,Z)$}
\psfrag{mx}{$\pi_X$}
\psfrag{my}{$\pi_Y$}
\psfrag{mz}{$\pi_Z$}
\includegraphics[scale=0.6]{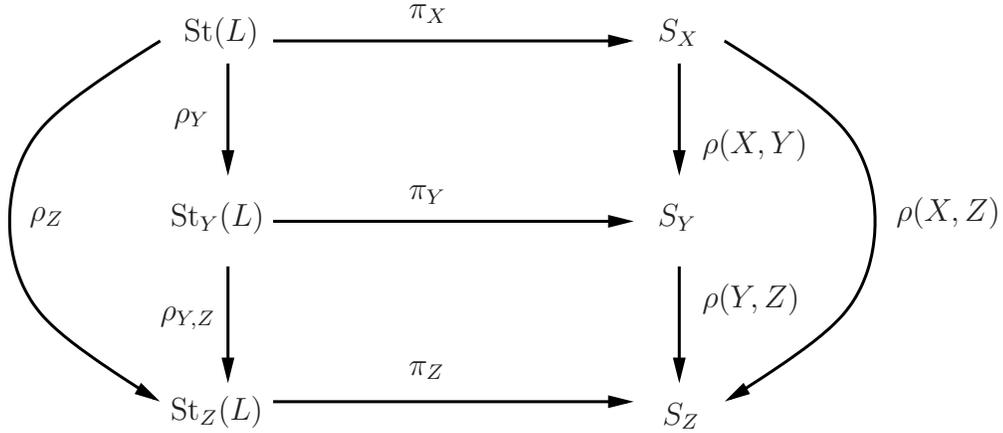}
\end{center}
\caption{Maps defined on subgroups of $\St(L)$}\label{fig:maps}
\end{figure}
\subsection{The structure of $\cSt(L)$}

First we examine the structure of $\St_Y(L)$ and $S_Y$,
for an arbitrary closed
set $Y$.
\begin{defn}\label{defn:SYsubs}
Let $Y\in L$ and $Y=\cup_{i=1}^m [v_i]_\perp$,
where
$Y^{\min}=\{v_1,\ldots ,v_m\}$ with $v_1\prec \cdots \prec v_m$.
Denote by $D_Y$ the set of block diagonal elements of $S_Y$;
with diagonal blocks
$A_1, ..., A_m$ such that $A_i\in \GL(|[v_i]_\perp|,\ZZ)$. Denote
by $U_Y$ the subset of $S_Y$ consisting of matrices for  which
$A_i$ is the identity matrix of $\GL(|[v_i]_\perp|,\ZZ)$,
for $i=1,\ldots ,m$.
\end{defn}

\begin{lem}\label{lem:SYstr}
Let $Y\in L$ and $Y=\cup_{i=1}^m [v_i]_\perp$,
where
$Y^{\min}=\{v_1,\ldots ,v_m\}$ with $v_1\prec \cdots \prec v_m$.
Both $U_Y$ and $D_Y$ are subgroups of $S_Y$ and
$S_Y=U_Y\rtimes D_Y$. Moreover
\[D_Y=\prod_{i=1}^m G[v_i]_\perp.
\]
\end{lem}
\begin{proof}
Let $A\in S_Y$ with diagonal blocks $A_1,\ldots ,A_m$ and define
$A_D$ to be the block-diagonal matrix with diagonal blocks $A_1,\ldots ,A_m$,
and let $d$ be the map sending $A$ to $A_D$.
Then $(A^d)^{\pi_Y^{-1}}$ is clearly an element of $\St_Y(L)$ and
so $A^d\in S_Y$. Hence $A^d\in D_Y$ and
$d$ is a surjective map from $S_Y$ to $D_Y$. If $B$ is also in
$S_Y$ and has diagonal blocks $B_1,\ldots, B_m$ then, as in the proof
of \ref{lem:Ssemi}, $AB$ has blocks $A_1B_1,\ldots, A_mB_m$, so
$d$ is a homomorphism and $D_Y$ is a group.
If $A\in S_Y$ then $A\in \ker(d)$ if and only if every diagonal block
of $A$ is an identity matrix. Hence $\ker(d)=U_Y$. If $i$ is the inclusion
of $D_Y$ in $S_Y$ then $id$ is the identity map on $D_Y$ and so
$S_Y= U_Y\rtimes D_Y$, as claimed.
That $D_Y=\prod_{i=1}^m G[v_i]_\perp$
is immediate from the definitions.
\end{proof}
%%%%%%%%%%%%%%%%%%5
In the case where $Y$ is the closure of a single element of $X$ we have
the following corollary of the above results.
\begin{cor}%[{\bf Key Lemma}]
\label{cor:Stkey}
\be
\item\label{it:Stkey1} $[\phi_x]\in S_x$, for all $\phi\in \St(L)$.
\item\label{it:Stkey2}  There are automorphisms $\phi_{x,s}$ and $\phi_{x,u}$ of $G(\cl(x))$
such that $\phi_x=\phi_{x,s}\phi_{x,u}$ and $[\phi_{x,s}]$ is the block-diagonal
matrix with diagonal blocks $A_1,\ldots ,A_m$ and $[\phi_{x,u}]$ is an
upper unitriangular
matrix (the unipotent part of $[\phi_x]$).
\item\label{it:Stkey3}  Given $A\in S_x$ there exists $\psi\in \St(L)$ such that $[\psi_x]=A$.
\item\label{it:Stkey4}  If $y<_L x$ then $[\phi_y]=M([\phi_x],\cl(y))$.
\item\label{it:Stkey5}  The set $S_x$ is a group.
\ee
\end{cor}
\begin{proof}
\ref{it:Stkey1}, \ref{it:Stkey3} and \ref{it:Stkey4} follow from
Lemma \ref{lem:StLY}. \ref{it:Stkey5} follows from Lemma \ref{lem:min}
and \ref{it:Stkey2}
follows from Lemma \ref{lem:SYstr}.
\end{proof}
\begin{thm}
$\cSt(L)=\LInn(G) \rtimes St(L)$.
\end{thm}
\begin{proof}
First we show that $\LInn(G) \subseteq\cSt(L)$.
As $\LInn(G)$ is generated by automorphisms of the form given
in Definition \ref{defn:elia} it suffices to show that $\phi=\a_C(y)$
belongs to $\cSt(L)$,
 where $y\in X$ and $C$ is a connected component of the full
subgraph on $X\backslash y^\perp$.
Let $V\in L$, so $V=T^\perp$, for some $T\in L$.
If $y\in V$ then $G(Y)^\phi=G(V)$, so assume
$y\notin V$. Let $v_1,v_2\in V$. If $v_1\in y^\perp$ then $v_1^\phi=
v_2^\phi$, so assume $v_i\notin y^\perp$, for $i=1,2$.
Now $y\notin V$ implies there exists some $t\in T$ such that $[y,t]\neq 1$.
As $v_i\in T^\perp$ and $v_i\notin y^\perp$ it follows that $v_1,v_2$ and
$t$ lie in a connected component of $\G(X\backslash y^\perp)$. In particular
$v_1^\phi=v_2^\phi$. Therefore either $G(V)^\phi=G(V)$ or $G(V)^\phi
=G(V)^y$ and so $\phi\in \cSt(L)$ as required.

Next we show that $\LInn(G) \lhd \cSt(L)$. It suffices to show
that $\theta^{-1}\phi\theta \in\LInn(G)$, where $\phi$ is defined as above.
Let $x\in X$ and let $g\in G$ such that $G(\cl(x))^\theta=G(\cl(x))^g$.
Then $G(\cl(x))^{\theta^{-1}}=G(\cl(x))^h$, where $h=(g^{-1})^{\theta^{-1}}$.
Thus $x^{\theta^{-1}}=w^h$, for some $w\in \cl(x)$, so $x=(w^\theta)^{h\theta}$
and $x^g=w^\theta$. Now $x^{\theta^{-1}\phi\theta}=
(w^h)^{\phi\theta}=((w^\phi)^{h^\phi})^\theta$ so
\[x^{\theta^{-1}\phi\theta}
=\left\{
\begin{array}{ll}
((w^y)^{h^\phi})^\theta=(w^\theta)^{y^\theta h^{\phi\theta}}=
x^{gy^\theta h^{\phi\theta}}, &
\textrm{ if } \cl(x)\subseteq C\cup y^\perp\\
(w^{h^\phi})^\theta=(w^\theta)^{h^{\phi\theta}}=x^{gh^{\phi\theta}}, &
\textrm{ otherwise}
\end{array}
\right.
.
\]
Therefore $\theta^{-1}\phi\theta$ is a conjugating automorphism.

Now we shall show that any element $\theta\in \cSt(L)$ can be expressed
as $\theta=\tau\phi$, for some $\tau\in \LInn(G)$ and $\phi \in \St(L)$.
Fix $\theta \in \cSt(L)$ and, for all $Y\in L$ fix $g_Y\in G$
such that $G(Y)^\theta=G(Y)^{g_Y}$. Without loss of generality
we may choose $g_Y$ so that none of its left divisors belong to $G(Y)$
or to $C_G(Y)=G(Y^\perp)$.
Given
two non-empty closed sets $Y, Z\in L$, with
$Y\subseteq Z$ we claim that
$g_{Y}g_Z^{-1}=ab$ where $a\in G(Z)$ and $b\in C_G(Y)$.
To see this suppose that $u\in G(Y)$ and let
$r\in G(Y)$, $s\in G(Z)$ such that $u^\theta = r^{g_Y}=s^{g_Z}$.
From \cite[Corollary 2.4]{DKR2} and the choice of $g_Y$ and $g_Z$ there
exist $c, c^\prime,d_1,d_1^\prime, d_2, d_2^\prime,v\in G$ such that
$g_Y=c\circ d_2$, $g_Z=c^\prime \circ d_2^\prime$,
$r=d_1^{-1}\circ v\circ d_1$, $s=d_1^{\prime -1} \circ v\circ d_1^\prime$,
and with $d=d_1\circ d_2$ and $d^\prime=d_1^\prime \circ d_2^\prime$,
$r^{g_Y}=d^{-1}\circ v\circ d$
and $s^{g_Z}=d^{\prime^{-1}}\circ v\circ d^\prime$.

By definition of $\theta$, for
$x\in Y$ there exists  $u\in G(Y)$ such that $u^\theta =x^{g_Y}$. We
may then take $r=x$ and $s\in G(Z)$ such that $u^\theta = s^{g_Z}=x^{g_Y}$.
In this case we shall have $r=x=v$ and so $d_1=1$ and from loc. cit.
Corollary 2.4 $c,c^\prime \in C_G(x)$. Allowing $x$ to range over $Y$
we see that $c,c^\prime \in C_G(Y)$ and by choice of $g_Y$
it follows that $c=1$ and $g_Y=d_2$. Now, with $r=x\in Y$ again we have
$d_2^{-1}\circ x\circ d_2=r^{g_Y}=s^{g_Z}=d^{\prime -1}\circ x \circ d^\prime,$
so $d^\prime =d_2$. As $\a(d_1^\prime)\subseteq \a(s)\subseteq Z$
 we have $d_1^\prime \in G(Z)$.  If $d_1^\prime$ has a left divisor in
$G(Y)$ then
so does $g_Y$ and  $c^\prime \notin C_G(Z)$ as it's a left divisor of
 $g_Z$. This completes the proof of the claim as
$g_Yg_Z^{-1}=d_1^\prime c^{\prime -1}$.

Next we use $\theta$ to
construct a homomorphism from $G$ to itself and subsequently
show that this homomorphism is an element of $\St(L)$.
Let $x\in X$; so $G(\cl(x))^\theta=G(\cl(x))^{g_x}$. Then there
exists $u_x\in G(\cl(x))$ such that $x^\theta=u_x^{g_x}$. Define
 a map $\phi:X\rightarrow G$ by $x^\phi =u_x$, for all $x\in X$. Suppose
$x,y\in X$ with $[x,y]=1$.  Then $x,y\in x^\perp\cap y^\perp$ so
$\cl(x)\cup\cl(y)\subseteq x^\perp\cap y^\perp$.
Let $Z=x^\perp\cap y^\perp$ and so $G(Z)^\theta=G(Z)^{g_Z}$.
As $\cl(x)\subseteq Z$ we have, from the above, $g_x=abg_Z$, with
$a\in G(Z)$ and $b\in C_G(\cl(x))=G(x^\perp)$. Thus $ab\in G(x^\perp)$
and $u_x^{ab}=u_x$. Hence $u_x^{g_x}=u_x^{g_Z}$ and similarly
$u_y^{g_y}=u_y^{g_Z}$. As $\theta$ is an automorphism we have
$1=[u_x^{g_x},u_y^{g_y}]=[u_x,u_y]^{g_Z}$, so $[u_x,u_y]=1$.
Therefore $\phi$ extends to an endomorphism of $G$.

The next step is to show that $\phi$ is surjective. To this end suppose
that $y,z\in X$ and $\cl(y)\subseteq \cl(z)$. If $u\in\cl(y)$ then
$[u,v]=1$, for all $v\in \cl(z)$, as $\cl(z)\subseteq z^\perp\subseteq
y^\perp$. We have $g_y=abg_z$, where $a\in G(\cl(z))$ and
$b\in G(y^\perp)$. Hence $u_y^{ab}=u_y$ and so
$u_y^{g_y}=u_y^{g_z}$. Now let $x\in X$. As $G(\cl(x))^\theta=G(\cl(x))^{g_x}$
there exists $w\in G(\cl(x))$ such that $w^\theta=x^{g_x}$. Assume
$w=y_1^{\e_1}\cdots y_n^{\e_n}$, for some $y_i\in \cl(x)$ and $\e_i=\pm 1$.
Let $u_i=u_{y_i}$. Then $\cl(y_i)\subseteq \cl(x)$ so
$y_i^\theta=u_i^{g_x}$, from the preceding
argument, and $w^\theta=(u_1^{\e_1}\cdots u_n^{\e_n})^{g_x}=x^{g_x}$.
Hence $w^\phi=u_1^{\e_1}\cdots u_n^{\e_n}=x$ and $\phi$ is surjective.

To show that $\phi$ is injective consider the automorphism $\theta^{-1}$
and let $h_x=(g_x^{-1})^{\theta^{-1}}$, for all $x\in X$. Choose,
for all $x\in X$, an element $k_x\in G$ and $v_x\in G(\cl(x))$
such that $G(\cl(x))^{\theta^{-1}}=G(\cl(x))^{k_x}$,
$x^{\theta^{-1}}=v_x^{k_x}$ and $k_x$ has no
left divisor in $G(\cl(x))$. Then, as in the
case of $\theta$ and $\phi$ above, the map $\overline\phi:x\rightarrow v_x$
extends to an endomorphism of $G$. Moreover $h_x=j_xk_x$, for some
$j_x\in G(x^\perp)$. Suppose that $u_x=y_1^{\e_1}\cdots y_n^{\e_n}$,
where $y_i\in \cl(x)$. Write $v_i=v_{y_i}
=y_i^{\overline{\phi}},$ for $i=1,\ldots, n$.
Then, from the above, $y_i^{\theta^{-1}}=v_i^{k_x}=
 v_i^{h_x}$, as $\cl(y_i)\subseteq \cl(x)$,
$v_i\in \cl(y_i)$
 and $j_x\in G(x^\perp)$.
 Now $x=x^{\theta\theta^{-1}}=(u_x^{g_x})^{\theta^{-1}}=
(u_x^{\theta^{-1}})^{h_x^{-1}}=
(v_1^{\e_1h_x}\cdots v_n^{\e_nh_x})^{h_x^{-1}}=
v_1^{\e_1}\cdots v_n^{\e_n},$ so
$x^{\phi\overline{\phi}}=u_x^{\overline{\phi}}=v_1^{\e_1}\cdots v_n^{\e_n}
=x$. It follows that $\phi$ is a bijection and hence is an automorphism.
By definition $\phi$ maps $G(\cl(x))$ to itself, for all $x\in X$, and
so belongs to $\St(L)$.

Now define $l_x=g_x^{\phi^{-1}}$, for all $x\in X$. Then $x^\theta=u_x^{g_x}=(x^\phi)^{l_x^\phi}=(x^{l_x})^\phi$, so $\t=\theta\phi^{-1}$ is a conjugating automorphism.
Note that if $\rho\in \LInn(G)\cap \St(L)$ then $x^\rho=x^{w_x}$, for some $w_x\in G$, and $G(\cl(x))^\rho=G(\cl(x))$; so $x^{w_x}\in G(\cl(x))$.
It follows that $w_x\in G(x^\perp)$ so $x^\rho=x$ and $\rho$ is the identity map. Hence $\LInn(G)\cap \St(L)=\{1\}$. Now suppose that $\t,\t^\prime\in \LInn(G)$ and $\phi, \phi^\prime\in \St(L)$. Then $\t\phi=\t^\prime \phi^\prime$ implies $\t^{\prime -1}\t=\phi^\prime\phi \in \LInn(G)\cap \St(L)$, so $\t=\t^\prime$ and $\phi=\phi^\prime$. What we have shown is that every element $\theta \in \cSt(L)$ can be uniquely expressed as $\theta =\t\phi$ with $\t\in \LInn(G)$ and $\phi\in \St(L)$. The theorem now follows.
\end{proof}

\end{document}